\documentclass[12pt,psfig,reqno]{amsart}
\usepackage{amsfonts}
\pdfoutput=1
\usepackage{amscd}
\usepackage{mathrsfs}
\usepackage{}
  \usepackage{hyperref}
\usepackage{amssymb,amsfonts,latexsym}
\usepackage{graphics,verbatim}
\usepackage{graphicx}
\usepackage{mathabx}
\usepackage[usenames]{color} 
\usepackage{cite}
\usepackage{verbatim}
\makeatletter

\renewcommand\section{\@startsection {section}{1}{\z@}%
                                   {-3.5ex \@plus -1ex \@minus -.2ex}%
                                   {2.3ex \@plus.2ex}%
                                   {\centering\normalfont\bf}}

 \numberwithin{equation}{section}
\numberwithin{equation}{section}

\newcommand{\I}{\int_{0}^{1}}      \newcommand{\Id}{\int_{\mathbb{D}}}
\newcommand{\sk}{\sum_{k=0}^{\infty}}   \newcommand{\ap}{A^{p}}
\newcommand{\sn}{\sum_{n=0}^{\infty}}              
  \newcommand{\fap}{||f||_{A^{p}}}             \newcommand{\ii}{\mathcal {I}} \newcommand{\iu}{\mathcal {I}_{\mu_{\alpha+1}}}
\newcommand{\hd}{H(\mathbb{D})}            
\newcommand{\dd}{\mathbb{D}}           
\newcommand{\B}{\mathcal {B}}
\newcommand{\hur}{\mathcal {H}_{\mu}^{\alpha}} \newcommand{\hu}{\mathcal {H}_{\mu}}
\numberwithin{equation}{section}

\theoremstyle{plain}
\newtheorem{thm}{Theorem}[section]
\newtheorem{lemma}[thm]{Lemma}

\newtheorem{cor}[thm]{Corollary}

\newtheorem{re}[thm]{Remark}
\newenvironment{proof of Theorem 1.3}{{\noindent\it Proof of Theorem 1.3}\quad}{\hfill $\square$\par}
\newenvironment{proof of Theorem 1.1}{{\noindent\it Proof of Theorem 1.1}\quad}{\hfill $\square$\par}
\begin{document}
\title{Generalized  Hilbert operator acting on   Bergman  spaces }
\author{Pengcheng Tang$^{*,1}$  \ \ Xuejun  Zhang$^{2}$  }

\address{$^{1}$  College of Mathematics and Statistics, Hunan University of Science and Technology, Xiangtan, Hunan 411201, China}

\address{$^{2}$  College of Mathematics and Statistics, Hunan Normal University, Changsha, Hunan 410006, China}
\email{1228928716@qq.com}

\date{}

\thanks{$^*$ Corresponding author.\\
 \ \ \ \ \ \ Email:www.tang-tpc.com@foxmail.com(Pengcheng Tang)\   xuejunttt@263.net(Xuejun Zhang)\\
\ \ \ \ \ \ The first author was supported by the Scientific Research Fund of Hunan Provincial Education Department (NO. 24C0222) and the second author was supported by the Natural Science Foundation of Hunan Province (No. 2022JJ30369).
 }

\begin{abstract}
Let $\mu$ be a positive Borel measure on $[0,1)$ and $f(z)=\sum_{n=0}^{\infty}a_{n}z^{n}\in \hd$.  For $\alpha>-1$,  the generalized  Hilbert operator is defined as follows:
 $$\hur(f)(z)=\sum_{n=0}^{\infty}\frac{\Gamma(n+1+\alpha)}{\Gamma(n+1)\Gamma(\alpha+1)}\left(\sum_{k=0}^{\infty}\mu_{n,k}a_{k}\right)z^{n}, \ \ \  z\in \dd,$$
where, for $n\geq 0$, $\mu_n$ denotes the $n$-th  moment of the measure  $\mu$, and $\mu_{n,k}=\mu_{n+k}$.

 In  this paper, we characterize
the measures $\mu$ for which  $\hur$ is a bounded operator acting   from   Bergman space $A^{p}(0<p<\infty)$ into $ A^{q}(q\geq 1)$. We determine the Hilbert-Schmidt  class  on $A^{2}$ for all $\alpha>-1$. In addition, we  also study the boundedness (resp. compactness) of $\hur$ acting from $A^{p}$ to the Bloch space.
\end{abstract}
\maketitle
\section{Introduction }\label{sec1}
\ \  \ Let $\mathbb{D}=\{z\in \mathbb{C}:\vert z\vert <1\}$ denote the open unit disk of the complex plane $\mathbb{C}$ and $H(\mathbb{D})$ denote the space of all
analytic functions in $\mathbb{D}$. $H^{\infty}(\dd)$ denote the set  of bounded analytical functions.

The Bloch space $\mathcal {B}$ consists of those functions $f\in  H(\mathbb{D}) $ for which
$$
\vert \vert f\vert \vert _{\mathcal {B}}=\vert f(0)\vert +\sup_{z\in \mathbb{D}}(1-\vert z\vert ^{2})\vert f'(z)\vert <\infty.
$$

The little Bloch space $\mathcal {B}_{0}$ is the closed subspace of  $\mathcal {B}$ consists of those
functions $f\in  H(\mathbb{D}) $ such that
$$
\lim_{\vert z\vert \rightarrow 1^{-}}(1-\vert z\vert ^{2})\vert f'(z)\vert =0.
$$

 For $0<p<\infty$, the Bergman  space $A^{p}$ consists of    $f\in H(\mathbb{D})$ such that $$||f||^{p}_{A^{p}}=\int_{\mathbb{D}}|f(z)|^{p}dA(z)<\infty,$$
where $dA(z)=\frac{dxdy}{\pi}$ is the normalized area measure on $\mathbb{D}$. We refer to \cite{hb3,b7} for the notation and results about
Bergman spaces and Bloch spaces.

 Let $f(z)=\sn a_{n}z^{n}\in H(\dd)$. For any complex parameters $t$ and $s$ with the property that
neither $1+t$ nor $1+t+s$ is a negative integer, we define the  fractional differential operator $R^{t,s}$ and the fractional integral operator $R_{t,s}$ as follows:
$$
R^{t,s}f(z)=\sum_{n=0}^{\infty}\frac{\Gamma(2+t)\Gamma(n+2+t+s)}{\Gamma(2+t+s)\Gamma(n+2+t)}a_{n}z^{n},
$$
$$
R_{t,s}f(z)=\sum_{n=0}^{\infty}
\frac{\Gamma(2+t+s)\Gamma(n+2+t)}{\Gamma(2+t)\Gamma(n+2+s+s)}a_{n}z^{n}.
$$

Let $\mu$ be a finite Borel measure on $[0,1)$ and $n\in \mathbb{N}$. We use $\mu_{n}$ to denote the sequence of order $n$ of $\mu$, that is,  $\mu_{n}=\int_{[0,1)}t^{n}d\mu(t)$. Let $\mathcal {H}_{\mu}$ be the Hankel matrix $(\mu_{n,k})_{n,k\geq 0}$ with entries $\mu_{n,k}=\mu_{n+k}$. The matrix $\mathcal {H}_{\mu}$ induces  an operator on $ H(\mathbb{D}) $ by its action on the Taylor coefficients $:a_{n}\rightarrow\displaystyle{\sum_{k=0}^{\infty}}\mu_{n,k}a_{k},\ \ n=0,1,2,\cdots $.

If $f(z)=\displaystyle{\sum_{n=0}^{\infty}}a_{n}z^{n} \in  H(\mathbb{D})$, the generalized Hilbert operator is defined as follows:
$$
\mathcal {H}_{\mu}(f)(z)=\sum_{n=0}^{\infty}\left(\sum_{k=0}^{\infty}\mu_{n,k}a_{k}\right)z^{n}.
$$
If  $\mu$ is the Lebesgue measure on $[0, 1)$, then  $\mathcal {H}_{\mu}$ reduce to the classic Hilbert operator $\mathcal {H}$.
 It's known that the generalized Hilbert operator $\hu$ is closely related to the following integral operator
 $$\ii_{\mu}(f)(z)=\I \frac{f(t)}{1-tz}d\mu(t).$$

The  study of the operator $\mathcal {H}_{\mu}$  on analytic function spaces was  initiated by Widom \cite{wid} in 1966. He proved that  the operator $\mathcal {H}_{\mu}$  is bounded (resp.  compact) on the Hardy space $H^{2}$ if and only if $\mu$ is a Carleson measure (resp.  vanishing  Carleson measure).  In 2000, Diamantopoulos and Siskakis  \cite{H8} studied the classic Hilbert operator  $\mathcal {H}$ on  Hardy spaces $H^{p}$ for $1<p<\infty$.  Subsequently, Diamantopoulus \cite{H9}
considered the boundedness of $\mathcal {H}$ on the Bergman spaces $\ap$ for $2<p<\infty$. In 2010, Galanopoulos and Pel\'{a}ez \cite{H14} investigated the boundedness of the operator $\hu$ on  Hardy space $H^{1}$ and on  Bergman space $A^{2}$.   In 2014, Chatzifountas, Girela and Pel\'{a}ez  \cite{H3}  systematically studied
the boundedness (resp.  compactness)  of $\mathcal {H}_{\mu}$  from $H^{p}$ to $H^{q}$ for $0<p,q<\infty$.  Since then, the study of  generalized Hilbert operator   has attracted the attention of many scholars.  See \cite{H1,H13,H5,H21,ls,H22,H23,H18,ba,t2,y1,y2,y3} for  more information on generalized Hilbert operator on  spaces  of analytic functions.

The  derivative or fractional derivative paly a basic role in the theory of  analytic function spaces.   This motivates us to consider the derivative or fractional derivative of the operator  $\hu$.
 In this paper, we consider the generalized  Hilbert operator as follows:
 $$\setlength{\abovedisplayskip}{3pt}\setlength{\belowdisplayskip}{3pt}
 \mathcal {H}^{\alpha}_{\mu}(f)(z)=\sum_{n=0}^{\infty}\frac{\Gamma(n+1+\alpha)}{\Gamma(n+1)\Gamma(\alpha+1)}\left(\sum_{k=0}^{\infty}\mu_{n,k}a_{k}\right)z^{n}, \ \ (\alpha>-1). $$
 
 The operator $\mathcal {H}^{\alpha}_{\mu}$ can be regarded as the fractional derivative of  $\mathcal {H}_{\mu}$, since it is easy to obtain  that  $\mathcal {H}^{\alpha}_{\mu}(f)=R^{-1,\alpha}\mathcal {H}_{\mu}(f)$. Hence, we may say that  $\mathcal {H}^{\alpha}_{\mu}$ is the  fractional derivative Hilbert operator. If $\alpha=0$, the operator $\mathcal {H}^{0}_{\mu}$ is just the $\hu$.
  If $\alpha=1$, then the operator $\mathcal {H}^{\alpha}_{\mu}$  is called Derivative-Hilbert operator. Naturally, the operator $\hur$ is  closely related to the
operator  $\iu$ which  is defined by
 $$\ii_{\mu_{\alpha+1}}(f)(z)=\I \frac{f(t)}{(1-tz)^{\alpha+1}}d\mu(t),\ \ \ (\alpha>-1).$$
If $\alpha=0$, then  $\ii_{\mu_{\alpha+1}}$ is just  the integral operator $\ii_{\mu}$.

The connection between  $\mathcal {H}_{\mu}$ and $\mathcal {H}^{\alpha}_{\mu}$ motivates us to consider the operator  $\hur$  for all $\alpha>-1$ in a unified manner. 
 The Bergman space  is one of the most basic  analytic function spaces. The operator $\hu$ on Bergman spaces has achieved a few  results, but there are still many unsolved problems. The following question is raised naturally.\\
\emph{Question: what are the necessary and sufficient conditions for $\hur: A^{p} \rightarrow A^{q}$$(0<p,q\leq\infty)$ to be a bounded (or compact) operator?}
Here, the  space $A^{\infty}$ should be interpreted as  the Bloch space $\B$ since the Bloch space $\B$ can be viewed as the limit case of $A^{p}$ as  $p\rightarrow +\infty$.

The study of the operator $\hur$ between Bergman spaces did not proceed  well even for $\alpha=0$.  In 2019, Girela and Merch\'{a}n \cite{H24} proved that $\hu$ is bounded on  $\ap(2<p<\infty)$ if and only if $\mu$ is a Carleson measure.
 Ye and Zhou solved  the problem for $\alpha=1$ and $1\leq q<p=\infty$ in \cite{H19}. They also solved the problem for $\alpha=1$, $p\leq q$ and $1\leq q\leq\infty$ in \cite{d1}. The first author of this paper  solved  the problem for all $\alpha>-1$ and $1\leq q<p=\infty$ in \cite{t1}.
 Aguilar-Hern\'{a}ndez,  Galanopoulos and Girela \cite{g1} considered the  operator $\hu$ on $\ap$ for $1\leq p<2$. They solved the problem for $\alpha=0$ and $p=q=1$. They also obtain  a  sufficient condition and  a necessary condition for the case   $1<p<2$.
  Recently, the authors  were  informed that Sun et al. solved the case $\alpha=0$ and $1<p=q<2$.
A complete solution for all indicators is far away.  In this paper, we will continue to explore this problem,   giving some results for certain cases.

This paper is organized as follows. In section 2, we list some of the lemmas that will be used.
In section 3, we mainly investigate the operator $\hur$ acting between the Bergman spaces. We  give  the  necessary and sufficient conditions such that the operator $\hur$ acting from $\ap(0<p\leq 1)$ into $A^{q}(q\geq 1)$ to be bounded. By means of the Marcinkiewicz interpolation theorem, we give the necessary and sufficient
conditions  for the  operator $\hur$ to be bounded on $A^{p}(1\leq p \leq 2)$. Finally, we  determine the Hilbert-Schmidt class on $A^{2}$ for all $\alpha>-1$.
 Section 4 is  devote to study the boundeness and compactness of $\hur$ acting from  Bergman space $A^{p}$ to  the Bloch sapce $\B$.
 By using the sublinear  generalized integral type Hilbert operator and a decomposition theorem of  the Bloch space $\B$, we have completely characterized the boundedness (resp. compactness) of  $\hur:\ap \rightarrow \B$.

 As mentioned previously,  the Carleson-type measures play a basic role in the studies of  the generalized Hilbert operator. Let $I\subset \partial \mathbb{D}$ be an arc, and  $\vert I\vert $  denote the length of $I$. The Carleson square $S(I)$ is defined as
$$S(I)=\{re^{i\vartheta}:e^{i\vartheta}\in I,\ 1-\frac{\vert I\vert }{2\pi}\leq r<1\}.$$

Let $\mu$ be a positive Borel measure on $\mathbb{D}$. For $0\leq \beta<\infty$ and $0<t<\infty$, we
say that $\mu$ is a $\beta$-logarithmic $t$-Carleson measure (resp.a vanishing $\beta$-logarthmic $t$-Carleson measure) if
$$
\sup_{ I \subset \partial\mathbb{D}}\frac{\mu(S(I))(\log\frac{2\pi}{\vert I\vert })^{\beta}}{\vert I\vert ^{t}}<\infty,\ \ \left(\mbox{resp.}\ \ \lim_{\vert I\vert \rightarrow0}\frac{\mu(S(I))(\log\frac{2\pi}{\vert I\vert })^{\beta}}{\vert I\vert ^{t}}=0\right).
$$
If $\beta=0$ and $t=1$, we  say  that $\mu$ is a Carleson measure. See \cite{H16} for more about  logarithmic type Carleson measure.

A positive Borel measure $\mu$ on $[0,1)$ can be seen as a Borel measure on $\mathbb{D}$ by identifying it with the measure $\overline{\mu}$ defined by
$$
 \overline{\mu}(E)=\mu(E\cap [0,1)), \ \ \mbox{for any Borel subset }\ E \ \ \mbox{of}\  \ \mathbb{D}.
$$

In this way, a positive Borel measure $\mu$ on $[0,1)$ is a $\beta$-logarithmic $t$-Carleson measure if and only if there exists a constant $M>0$ such that
$$
\mu([s,1))\log^{\beta}\frac{e}{1-s}\leq M(1-s)^{t},\ \ 0\leq s<1.
$$

 Throughout the paper, the letter $C$ will denote an absolute constant whose value depends on the parameters
indicated in the parenthesis, and may change from one occurrence to another. We will use
the notation $``P\lesssim Q"$ if there exists a constant $C=C(\cdot) $ such that $`` P \leq CQ"$, and $`` P \gtrsim Q"$ is
understood in an analogous manner. In particular, if  $``P\lesssim Q"$  and $ ``P \gtrsim Q"$ , then we will write $``P\asymp Q"$.
We shall use the notation that for any given $s>1$, $s'$ will denote the conjugate exponent of $s$, that is, $\frac{1}{s}+\frac{1}{s'}=1$, or $s'=\frac{s}{s-1}$.

\section{Preliminary Results}\label{sec2}


Mateljevi\'{c} and Pavlovi\'{c} \cite{L14} proved the following result.
\begin{lemma}\label{lm2.1}
Let $1<p<\infty$ and $\gamma>-1$. For a function $f\in \hd$  we define:
$$K_{1}(f):=\Id |f(z)|^{p}(1-|z|)^{\gamma}dA(z),\ \ \ K_{1}(f):=\sn 2^{-n(\gamma+1)}||\Delta_{n}f||^{p}_{p}.$$
Then, $K_{1}(f)\asymp K_{2}(f)$.
\end{lemma}
The following lemma can be found in \cite[ page 83 ]{hb3}.
\begin{lemma}\label{lm2.2}
Let  $f(z)=\sn a_{n}z^{n}\in\hd$, then
$$\sum_{n=1}^{\infty}n^{-1}|a_{n}|<\infty \ \ \Rightarrow f\in A^{1}\ \ \Rightarrow \sum_{n=1}^{\infty}n^{-2}|a_{n}|<\infty.$$
\end{lemma}

The following  lemma comes form  Theorem 3.5 in \cite{b7}.
\begin{lemma}\label{lm2.3}
Let $t>0$ and $f\in \hd$. If $s$ is a real parameter such that neither $1+s$ nor $1+s+t$ is a negative integer, then the following conditions are equivalent:
\\(a)\ $f\in \mathcal {B}$.
\\ (b)\ The function $(1-|z|^{2})^{t}R^{s,t}f(z)$ is bounded in $\dd$.

Moreover,
$$||f||_{\mathcal {B}}\asymp \sup_{z\in \dd} (1-|z|^{2})^{t}|R^{s,t}f(z)|.$$
\end{lemma}

The   following integral  estimates follows from  the  Proposition 1.4.10 in \cite{b3}.
\begin{lemma}\label{lm2.4}
 Let  $t$ be real and $\delta>-1$, $ z\in \dd$.   Then
the integral
$$
 J(z)=\int_{\dd}\frac{(1-|u|^{2})^{\delta}dA(u)}{|1-\overline{z}u|^{2+t+\delta}}\ \
$$
have the following asymptotic properties.
\\ (1) \ If  $t<0$, then $ J(z)\asymp 1$.
\\ (2) \ If  $t=0$, then $\displaystyle{J(z) \asymp
\log\frac{e}{1-|z|^{2}}}$.
\\ (3) \ If  $t>0$, then $  J(z) \asymp
 (1-|z|^{2})^{-t}$.
\end{lemma}

\begin{lemma}\label{lm2.5}
Let $\mu$ be a positive Borel measure on $ [0, 1)$, $\beta>0$, $\gamma>0$. Let $\tau $ be the Borel measure on $[0, 1)$ defined by
$$d\tau(t)=\frac{d\mu(t)}{(1-t)^{\gamma}}.$$
Then, the following statements are equivalent.
\\ (a)\ $\mu$ is a $\beta+\gamma$-Carleson measure.
\\(b)\ $\tau$ is a $\beta$-Carleson measure.
\end{lemma}
\begin{proof}

 $(b) \Rightarrow (a)$. Assume (b).  Then there exists a positive constant $C>0$ such that
$$\int_{t}^{1}\frac{d\mu(r)}{(1-r)^{\gamma}}\leq C(1-t)^{\beta},\ \ t\in[0,1).$$
Using this and the fact that the function $x \mapsto \frac{1}{(1-x)^{\gamma}}$ is increasing in $[0, 1)$, we
obtain
$$\frac{\mu([t,1))}{(1-t)^{\gamma}}\leq \int_{t}^{1}\frac{d\mu(r)}{(1-r)^{\gamma}}\leq C(1-t)^{\beta},\ \ t\in[0,1).$$
This shows that  $\mu$ is a $\beta+\gamma$-Carleson measure.

$(a) \Rightarrow (b)$. Assume (a). Then there exists a positive constant $C>0$ such that
$$\mu(t)\leq C(1-t)^{\beta+\gamma},\ \ t\in[0,1).$$

 Integrating by
parts and using the above inequality, we obtain
\[ \begin{split}
&\tau([t,1))=\int_{t}^{1}\frac{d\mu(x)}{(1-x)^{\gamma}}\\
&= \frac{1}{(1-t)^{\gamma}}\mu([t,1))-\lim_{x\rightarrow 1}\frac{1}{(1-x)^{\gamma}}\mu([x,1))+\gamma\int_{t}^{1}\frac{\mu([x,1))}{(1-x)^{\gamma+1}}dx\\
&= \frac{1}{(1-t)^{\gamma}}\mu([t,1))+\gamma\int_{t}^{1}\frac{\mu([x,1))}{(1-x)^{\gamma+1}}dx\\
& \lesssim (1-t)^{\beta}+\int_{t}^{1} (1-x)^{\beta-1}dx\lesssim  (1-t)^{\beta}.
  \end{split} \]
  Thus,  $\tau$ is a $\beta$-Carleson measure.
\end{proof}

The following lemma is useful for dealing with the compactness. It is a consequence of Theorem 3.1 in \cite{stt}.
\begin{lemma}\label{lm2.6}
Suppose that $T$ is a bounded operator from  $\ap$ into $Y$($Y=\mathcal {B} $ or $\ap$). Then $T$  is compact operator from $\ap$ into $Y$ if and only if  for any bounded sequence $\{h_{n}\}$ in $X$ which converges to $0$ uniformly on every compact subset of $\mathbb{D}$, we have
$\lim_{n\rightarrow \infty}||T(h_{n})|| _{Y}=0.$
\end{lemma}


\section{Generalized Hilbert operator acting  between  Bergman spaces   }
Let $0<p<\infty$, for each $f\in A^{p}$, it is well known that
\begin{equation}\label{eq4.1}
|f(z)|\lesssim\frac{||f||_{A^{p}}}{(1-|z|)^{\frac{2}{p}}}\ \ \mbox{for all}\ \ z\in \dd.
 \end{equation}

The following lemma provide a  sufficient condition for the operator  $\hur$ to be well defined on the Bergman space
$A^{p}$. For a proof, see  \cite[Theorem 3.3]{y2} or \cite[Theorem 2.1]{d1}.
\begin{lemma}\label{lm3.1}
Suppose $0<p<\infty$ and $\alpha>-1$. Let $\mu$ be a positive Borel measure on $[0, 1)$. Then  $\hur (f)$ is a well defined analytic function in $\dd$ for every $f\in \ap$ in any of the three following cases.
\\
(i) $\mu$ is a $\frac{2}{p}$ Carleson measure if $0<p\leq 1$.
\\
(ii) $\mu$ is a $\left(\frac{2-(p-1)^{2}}{p}\right)$ Carleson measure if $1\leq p\leq 2$.
\\
(iii) $\mu$ is a $\frac{1}{p}$ Carleson measure if $2<p<\infty$.
\\
Furthermore, in such cases we have that
$$\hur (f)=\int_{0}^{1}\frac{f(t)}{(1-tz)^{\alpha+1}}d\mu(t), \ z\in \dd, \  f\in \ap.$$
\end{lemma}

We start considering the case  $\alpha>-1$, $0<p\leq 1$ and $q=1$. We have the following result.
\begin{thm}\label{th3.2}
Suppose that $0<p\leq 1$  and  $\alpha>-1$, let  $\mu$ be a positive Borel measure on $[0, 1)$ which satisfies the conditions of Lemma \ref{lm3.1}. Then the following statements hold:\\
(a) If $\alpha> 1$, then  $\hur:A^{p}\rightarrow A^{1}$ is bounded if and only if $\mu$ is a $\frac{2}{p}+\alpha-1$ Carleson measure.
\\ (b) If  $-1<\alpha<1$, then  $\hur:A^{p}\rightarrow A^{1}$ is bounded if and only if $\mu$ is a $\frac{2}{p}$ Carleson measure.
\end{thm}
\begin{proof}
   (a).    If $\alpha>1$ and $\hur:A^{p}\rightarrow A^{1}$ is bounded. Let
$$f_{b}(z)=\frac{(1-b^{2})^{\frac{2}{p}}}{(1-bz)^{\frac{4}{p}}}, \ \ \ \  b\in(\frac{1}{2},1), \ z\in\dd.$$
Then it is easy to check that $f_{b}\in A^{p}$ and $\sup_{\frac{1}{2}<b<1}||f_{b}||_{A^{p}}\lesssim 1$.
Since $\hur(f_{b})\in A^{1}$, Lemma \ref{lm2.2} yields
$$\sum_{n=1}^{\infty}n^{-2}\frac{\Gamma(n+1+\alpha)}{\Gamma(n+1)\Gamma(\alpha+1)} \left|\I t^{n}f_{b}(t)d\mu(t)\right|\lesssim||\hur(f_{b})||_{A^{1}}\lesssim ||f_{b}||_{A^{p}} \lesssim 1 .$$
It is clear  that
 $$f_{b}(z)=\sk a_{k,b}z^{k}, \ \ \mbox{where}\ a_{k,b}=(1-b^{2})^{\frac{2}{p}}\frac{\Gamma(k+\frac{4}{p})b^{k}}{\Gamma(k+1)\Gamma(\frac{4}{p})},$$
and
   $$\sum_{n=1}^{\infty}n^{\alpha-2}b^{n} \asymp \frac{1}{(1-b)^{\alpha-1}}, \ \ \mbox{whenever}\   \alpha>1,\ b\in (\frac{1}{2},1).$$
 This together with Stirling's  formula imply that
\[ \begin{split}
1 &\gtrsim \sum_{n=1}^{\infty}n^{-2}\frac{\Gamma(n+1+\alpha)}{\Gamma(n+1)\Gamma(\alpha+1)} \left|\I t^{n}f_{b}(t)d\mu(t)\right|\\
&\gtrsim (1-b^{2})^{\frac{2}{p}} \sum_{n=1}^{\infty}n^{\alpha-2}\sk \frac{\Gamma(k+\frac{4}{p})b^{k}}{\Gamma(k+1)\Gamma(\frac{4}{p})}\int_{b}^{1}t^{n+k}d\mu(t)\\
&\gtrsim (1-b^{2})^{\frac{2}{p}}  \mu([b,1))\sum_{n=1}^{\infty}n^{\alpha-2}b^{n}\left(\sk \frac{\Gamma(k+\frac{4}{p})b^{2k}}{\Gamma(k+1)\Gamma(\frac{4}{p})}\right)\\
&\asymp  (1-b^{2})^{\frac{2}{p}}  \mu([b,1)) \frac{1}{(1-b^{2})^{\frac{4}{p}+\alpha-1}}\\
& = \frac{\mu([b,1))}{  (1-b^{2})^{\frac{2}{p}+\alpha-1}}.
  \end{split} \]
  Consequently, $\mu([b,1))\lesssim  (1-b^{2})^{\frac{2}{p}+\alpha-1}$ for all $\frac{1}{2}<b<1$. This shows that $\mu$ is a $\frac{2}{p}+\alpha-1$-Carleson measure.

    On the other hand, if  $\mu$ is a $\frac{2}{p}+\alpha-1$ Carleson measure. Then  $\displaystyle{\frac{d\mu(t)}{(1-t)^{\alpha-1}}}$ is a $\frac{2}{p}$-Carleson measure by Lemma \ref{lm2.5}. The assumption  of  the  measure $\mu$ insure that $\hur (f)=\iu (f)$ for every $f\in \ap$.  The well known Carleson embedding theorem for $A^{p}$ shows that
    \begin{equation}\label{eq4.2}
    \I |f(t)|\frac{d\mu(t)}{(1-t)^{\alpha-1}} \lesssim ||f||_{A^{p}} \ \ \mbox{for all } \ f\in A^{p},  0<p\leq 1. \end{equation}
   For each  $0\leq r<1$, $f\in \ap$ and $g\in \mathcal {B}_{o}$, we have
 \[ \begin{split}
   &  \ \ \ \ \  \int_{\mathbb{D}}\int_{[0,1)}\left| \frac{f(s)g(rz)}{(1-rsz)^{\alpha+1}}
\right| d\mu(s)dA(z)\\
&\leq \frac{1}{(1-r)^{2}}\int_{[0,1)}\frac{| f(s)|}{(1-s)^{\alpha-1}} d\mu(s)\int_{\mathbb{D}}|g(rz)| dA(z)\\
&\lesssim \frac{|| f||_{\ap}||g||_{\mathcal {B}}}{(1-r)^{2}}\int_{\mathbb{D}}\log\frac{e}{1-|z|}dA(z)\\
&\leq \frac{|| f||_{\mathcal {B}}||g||_{\mathcal {B}}}{(1-r)^{2}}<\infty,
 \end{split} \]
Let $g(z)=\sum_{n=0}^{\infty}b_{n}z^{n}$.  Then by Fubini's theorem and  a simple calculation through polar coordinate, we have that
 \[ \begin{split}
 &\ \ \ \ \Id \overline{\ii_{\mu_{\alpha+1}}(f)(rz)}g(rz)dA(z)\\
 &   = \I \sn \frac{\Gamma(n+1+\alpha)}{\Gamma(n+2)\Gamma(\alpha+1)}b_{n}(r^{2}t)^{n}\overline{f(t)}d\mu(t)\\
 &  =\I R^{0,\alpha-1} g(r^{2}t)  \overline{f(t)}d\mu(t).  \end{split} \]
    for all $0\leq r<1$, $f\in \ap$, $g\in \mathcal {B}_{0}$.


    For each  $0\leq r<1$, $f\in A^{p}$ and $g\in \mathcal {B}_{0}$, using Lemma \ref{lm2.3} and  (\ref{eq4.2}) we have
       \[ \begin{split}
     &\ \ \ \  \left| \Id \overline{\ii_{\mu_{\alpha+1}}(f)(rz)}g(rz)dA(z)\right|=\left|\I R^{0,\alpha-1} g(r^{2}t)  \overline{f(t)}d\mu(t)\right|\\
      & \lesssim  ||g||_{\mathcal {B}}\I |f(t)|\frac{d\mu(t)}{(1-r^{2}t)^{\alpha-1}}   \lesssim  ||g||_{\mathcal {B}}\I |f(t)|\frac{d\mu(t)}{(1-t)^{\alpha-1}} \\
     &  \lesssim ||f||_{A^{p}}||g||_{\mathcal {B}}.
                \end{split} \]
      Thus
     \begin{equation}\label{eq4.3}
     \lim_{r\rightarrow 1^{-}}\left| \Id \overline{\ii_{\mu_{\alpha+1}}(f)(rz)}g(rz)dA(z)\right|\lesssim ||f||_{A^{p}}||g||_{\mathcal {B}} .\end{equation}
    Since $(\mathcal {B}_{0})^{\star}\simeq A_{1}$  under the pairing
     $$\langle F, G\rangle=\lim_{r\rightarrow 1^{-}} \Id F(rz)\overline{G(rz)}dA(z), \ \ F\in \mathcal {B}_{0},  G\in A_{1}  .$$
      This together with (\ref{eq4.3}) imply that $\hur$  is a bounded operator  from $A^{p}$ into $A^{1}$.

      (b).  If $-1<\alpha<1$, then it is obvious that
      $$\sum_{n=1}^{\infty}n^{\alpha-2}b^{n} \asymp 1 \ \ \mbox{whenever}\ b\in (\frac{1}{2},1).$$
      Arguing as in the proof of  (a), we will obtain the necessity.

   On the other hand, suppose that $\mu$ is a $\frac{2}{p}$ Carleson measure, then  $\hur(f)$ is well defined  analytic function for each $f\in \ap$ by Lemma \ref{lm3.1}. Moreover, the Carleson embedding for $\ap$ implies that
$$\I |f(t)|d\mu(t)\lesssim \fap \ \mbox{for all}\  f\in \ap ,\ 0<p\leq 1 . $$
 Using Fubini's theorem and Lemma \ref{lm2.4}, we have
 \[ \begin{split}
||\hur(f)||_{A^{1}}&=\int_{\dd}\left|\int_{0}^{1}\frac{f(t)}{(1-tz)^{\alpha+1}}d\mu(t)
\right|dA(z)\\
 &\leq  \I |f(t)|\int_{\dd}\frac{dA(z)}{|1-tz|^{\alpha+1}}d\mu(t)\\
&\lesssim \I |f(t)|d\mu(t)\lesssim ||f||_{A^{p}}.
     \end{split} \]
     This implies that   $\hur:A^{p}\rightarrow A^{1}$ is bounded.\end{proof}

\begin{re}
For $\alpha=0$ and $p=1$, $\hu$ is bounded on $A^{1}$ if and only if $\mu$ is a  $2$-Carleson measure. This also was proved in \cite[Theorem 3]{g1}.
For $\alpha=1$, then $\mathcal {H}_{\mu}^{1}:A^{p}\rightarrow A^{1}$ is bounded if and only if $\mu$ is a 1-logarithmic $\frac{2}{p}$-Carleson measure. This result  was proved by Ye and Zhou  \cite[Theorem 3.2]{d1}.
\end{re}
As a consequence of Theorem  \ref{th3.2}, we may obtain the following results.
\begin{cor}
Let  $0<p\leq 1$ and let $\mu$ be a positive Borel measure on $[0, 1)$. If $\hur:A^{p}\rightarrow A^{1}$ is bounded for some $\alpha>-1$, then for any $-1<\alpha'<\alpha$, $\mathcal {H}_{\mu}^{\alpha'}:A^{p}\rightarrow A^{1}$ is bounded.
\end{cor}

\begin{cor} Suppose that $0<p<\infty$ and $\alpha>-1$. Let  $\mu$ be a positive Borel measure on $[0, 1)$. Then the following statements hold.
\\ (a)\ If $\alpha>1$ and $\displaystyle{\I \frac{d\mu(t)}{(1-t)^{\frac{2}{p}+\alpha-1}}<\infty}$,
then  $\hur:A^{p}\rightarrow A^{1}$ is bounded.
 \\ (b)\  If $-1<\alpha<1$ and $\displaystyle{\I \frac{d\mu(t)}{(1-t)^{\frac{2}{p}}}<\infty}$, then $\hur:A^{p}\rightarrow A^{1}$ is bounded.
\end{cor}
\begin{proof}
 (a)\ If $\alpha>1$ and $\I \frac{d\mu(t)}{(1-t)^{\frac{2}{p}+\alpha-1}}<\infty$,  for  every $f\in \ap$ and $g\in \mathcal {B}_{0}$,  by Lemma \ref{lm2.3} and (\ref{eq4.1}) we have that
   \[ \begin{split}
   \left|\I R^{0,\alpha-1} g(r^{2}t)  \overline{f(t)}d\mu(t)\right|
  & \lesssim ||f||_{A^{p}}||g||_{\mathcal {B}} \I \frac{d\mu(t)}{(1-r^{2}t)^{\alpha-1}(1-t)^{\frac{2}{p}}}\\
  & \lesssim  ||f||_{A^{p}}||g||_{\mathcal {B}}  \I \frac{d\mu(t)}{(1-t)^{\frac{2}{p}+\alpha-1}}\\
  & \lesssim  ||f||_{A^{p}}||g||_{\mathcal {B}}.
     \end{split} \]
     The rest of the proof is straightforward.

     (b)  If $-1<\alpha<1$ and $\I \frac{d\mu(t)}{(1-t)^{\frac{2}{p}}}<\infty$, then (\ref{eq4.1}) means that
     $$||\hur(f)||_{A^{1}}\lesssim\I |f(t)|d\mu(t)\lesssim ||f||_{A^{p}} \I \frac{d\mu(t)}{(1-t)^{\frac{2}{p}}}\lesssim ||f||_{A^{p}}.
     $$
     This finishes the proof. \end{proof}

\begin{re}
If  $\alpha> 1$  and $\displaystyle{\I \frac{d\mu(t)}{(1-t)^{\frac{2}{p}+\alpha-1}}<\infty}$ , then $\mu$ is a $\frac{2}{p}+\alpha-1$-Carleson measure. But the reverse does not hold. For example,  $d\mu(t)=(1-t)^{\frac{2}{p}+\alpha-2}dt$. It is easy to check that  $\mu$ is a $\frac{2}{p}+\alpha-1$-Carleson measure but $\displaystyle{\I \frac{d\mu(t)}{(1-t)^{\frac{2}{p}+\alpha-1}}=\infty}$.
\end{re}

Let us turn to the case  $0<p\leq 1$ and $1<q<\infty$. We have the following result.
\begin{thm}\label{th3.7}
Suppose that  $0<p \leq 1 <q<\infty$ and $\alpha>-1$. Let $\mu$ be a positive Borel measure on $[0, 1)$ which is a $\frac{2}{p}$-Carleson measure. Then the following statements hold:\\
(a) If  $\alpha+1-\frac{2}{q}>0$, then $\hur:A^{p}\rightarrow A^{q}$ is bounded if and only if $\mu$ is a $\frac{2}{p}+\frac{2}{q'}+\alpha-1$ Carleson measure.
\\(b) If   $\alpha+1-\frac{2}{q}<0$, then $\hur:A^{p}\rightarrow A^{q}$ is bounded.
\\ (c) If   $\alpha+1-\frac{2}{q}=0$ and $\mu$ is a $\frac{1}{q}$-logarithmic $\frac{2}{p}$-Carleson measure, then  $\hur:A^{p}\rightarrow A^{q}$ is bounded.

\end{thm}
\begin{proof}
 (a).  If  $\ii_{\mu_{\alpha+1}}:A^{p}\rightarrow A^{q}$ is bounded, then arguing as the  proof of Theorem \ref{th3.2} we may obtain  that $\mu$ is a $\frac{2}{p}+\frac{2}{q'}+\alpha-1$ Carleson measure. Thus, we only need to prove the sufficiency.
By using  the integral form of Minkowski's inequality and Lemma \ref{lm2.4}, we have
$$
 ||\hur(f)||_{A^{q}}\leq \left\{ \Id \left(\I \frac{|f(t)|}{|1-tz|^{\alpha+1}}d\mu(t)\right)^{q}dA(z)\right\}^{\frac{1}{q}}\ \ \ \ \ \  \ \ \ \ \ \ \ \  $$
\begin{equation}\label{eq4.4}
\leq \I|f(t)| \left(\Id\frac{dA(z)}{|1-tz|^{q(\alpha+1)}}\right)^{\frac{1}{q}}d\mu(t) \end{equation}
$$ \lesssim \I |f(t)| \frac{d\mu(t)}{(1-t)^{\alpha+1-\frac{2}{q}}}. \ \ \ \ \ \ \ \ \ \ \ \ \ \ \ \
 $$
If $\mu$ is a $\frac{2}{p}+\frac{2}{q'}+\alpha-1$ Carleson measure, then Lemma \ref{lm2.5} shows that $\frac{d\mu(t)}{(1-t)^{\alpha+1-\frac{2}{q}}}$ is a $\frac{2}{p}$-Carleson measure. Hence we have that
$$ \I |f(t)| \frac{d\mu(t)}{(1-t)^{\alpha+1-\frac{2}{q}}}\lesssim ||f||_{A^{p}}, \ \ \ 0<p\leq 1.$$
This means that $\hur:A^{p}\rightarrow A^{q}$ is bounded.

(b).  If   $\alpha+1-\frac{2}{q}<0$, by Lemma \ref{lm2.4} and (\ref{eq4.4}) we have
$$ ||\ii_{\mu_{\alpha+1}}(f)||_{A^{q}}\lesssim \I |f(t)| d\mu(t)\lesssim ||f||_{A^{p}}.
$$
This implies  $\hur:A^{p}\rightarrow A^{q}$ is bounded.

(c). If   $\alpha+1-\frac{2}{q}=0$, Lemma \ref{lm2.4} and  and (\ref{eq4.4}) show that
$$ ||\hur(f)||_{A^{q}}\lesssim \I |f(t)|\left(\log\frac{e}{1-t}\right)^{\frac{1}{q}} d\mu(t).$$
Since $\mu$  is a $\frac{1}{q}$-logarithmic $\frac{2}{p}$-Carleson measure,  it follows from Lemma 2.5 in \cite{H5}  that $\left(\log\frac{e}{1-t}\right)^{\frac{1}{q}} d\mu(t)$ is a $\frac{2}{p}$-Carleson measure. Hence,
$$ \I |f(t)|\left(\log\frac{e}{1-t}\right)^{\frac{1}{q}} d\mu(t)\lesssim ||f||_{A^{p}}.$$
This finishes the proof.
\end{proof}
\begin{thm}\label{th3.8}
Suppose that $1\leq p \leq 2$ and $\alpha>1$. Let $\mu$ be a positive Borel measure on $[0, 1) $ which  is a $\frac{2-(p-1)^{2}}{p}$ Carleson measure.  Then
 $\hur$ is a  bounded  operator on $A^{p}$ if and only if $\mu$ is an $\alpha+1$ Carleson measure.
\end{thm}
\begin{proof}
If  $\mu$ is a $\frac{2-(p-1)^{2}}{p}$-Carleson measure, then  Lemma \ref{lm3.1} shows that $\mathcal {H}_{\mu}^{\alpha}$ is well defined on $A^{p}$  and $\mathcal {H}_{\mu}^{\alpha}(f)=\ii_{\mu_{\alpha+1}}(f)$ for all $f\in \ap \ (1\leq p\leq 2)$.

If $p=1$,  it follows from Theorem \ref{th3.2} that $\hur$  is a bounded operator on $ A^{1}$ if
and only if $\mu$ is an $\alpha+1$-Carleson measure.

 If $p=2$ and $\mu$ is an $\alpha+1$ Carleson measure. Let $f(z)=\sn a_{n}z^{n}\in A^{2}$, then $||f||^{2}_{A^{2}}=\sn \frac{|a_{n}|^{2}}{n+1}$.  Since $\mu$ is an $\alpha+1$ Carleson measure, we have
\begin{equation}\label{eq4.5}
|\mu_{n,k}|=|\mu_{n+k}|\lesssim \frac{1}{(n+k+1)^{\alpha+1}}. \end{equation}
   Using (\ref{eq4.5})  and the Hilbert's  inequality,  we  have
     \[ \begin{split}
     ||\mathcal {H}_{\mu}^{\alpha}(f)||^{2}_{A^{2}}
     & \asymp \sn (n+1)^{2\alpha-1}\left|\sk \mu_{n,k}a_{k}\right|^{2}\\
     & \lesssim \sn (n+1)^{2\alpha-1}\left(\sk \frac{|a_{k}|}{(n+k+1)^{\alpha+1}}\right)^{2}\\
     & \leq \sn \left(\sk \frac{|a_{k}|}{(n+k+1)^{\frac{3}{2}}}\right)^{2}\\
     & \leq \sn  \left(\sk \frac{|a_{k}|}{(n+k+1)(k+1)^{\frac{1}{2}}}\right)^{2}\\
     & \lesssim  \sk \frac{|a_{k}|^{2}}{k+1}.
       \end{split} \]
   The last step above used the Hilbert's inequality. Thus   $\hur$ is a  bounded  operator on $A^{2}$.
    The complex interpolation theorem (see Theorem 2.34 in \cite{b7}) shows that                                                                                                                                                                                           $$A^{p}=[A^{1},A^{2}]_{\theta}, \ \ \ \mbox{if}\ 1<p<2\  \mbox{and}\  \theta=2-\frac{2}{p}. $$
This implies that $\hur$ is a  bounded  operator on $A^{p}(1\leq p \leq 2)$.  The proof of the  necessity is similar to  that of Theorem \ref{th3.2}, we omit the details here.
\end{proof}
In \cite{H14},  the authors proved that $\mathcal {H}_{\mu}$ is a Hilbert-Schmidt operator on $A^{2}$ if and only if
$$\I \frac{\mu([t,1))}{(1-t)^{2}}\log \frac{e}{1-t}d\mu(t)<\infty. $$

We recall that an operator $T$ on a separable Hilbert space
$X$ is a Hilbert-Schmidt operator if for an orthonormal basis $\{e_{k}: k=0,1,2,\cdots\}$ of
$X$  the sum $\sum_{k=0}^{\infty}||T(e_{k})||^{2}_{X}$  is finite. The finiteness of this sum does not depend on
the basis chosen.   Here, we consider the analogous problem on $A^{2}$  for  $\alpha\neq0$.

\begin{thm}
Let  $\mu$ be a positive Borel measure on $[0, 1) $ which  is a $\frac{1}{2}$ Carleson measure. Then the following statements  hold.
  \\ (a) If $\alpha>0$, then $\hur$ is a Hilbert-Schmidt operator on $A^{2}$
 if and only if
 \begin{equation}\label{eq4.6}
 \I \frac{\mu([t,1))}{(1-t)^{2+2\alpha}}d\mu(t)<\infty. \end{equation}
 (b) If $-1<\alpha<0$, then $\hur$ is a Hilbert-Schmidt operator on $A^{2}$.
\end{thm}
\begin{proof}
(a). Note that  $\mathcal {H}_{\mu}^{\alpha}(f)$ is well defined on  $ A^{2}$ by the  assumption of the measure $\mu$. Take the orthonormal basis $\{e_{k}\}_{k\geq 0}=\{(k+1)^{\frac{1}{2}}z^{k}\}_{k\geq 0}$, we have  that
  \[ \begin{split}
 \sk ||\mathcal {H}_{\mu}^{\alpha}(e_{k})||_{A ^{2}}^{2}
 &\asymp \sk (k+1)\sn (n+1)^{2\alpha-1}\mu_{n,k}^{2}\\
  & = \sk (k+1)\sn  (n+1)^{2\alpha-1}\I \I (ts)^{n+k}d\mu(s)d\mu(t)\\
  &\asymp \I \I \frac{1}{(1-ts)^{2+2\alpha}}d\mu(s)d\mu(t)\\
  & = 2 \I\int_{t}^{1} \frac{1}{(1-ts)^{2+2\alpha}}d\mu(s)d\mu(t)\\
  & \asymp \I \frac{\mu([t,1))}{(1-t)^{2+2\alpha}}d\mu(t).
      \end{split} \]
 So the operator  $\hur$  is Hilbert-Schmidt on $A^{2}$ if and only if (\ref{eq4.6}) holds.

(b). If $-1<\alpha<0$, then $\sn (n+1)^{2\alpha-1} \asymp1$. The rest of the  proof is obvious.
\end{proof}

\begin{cor}
Let  $\mu$ be  a positive Borel measure on $[0, 1) $. If   $\hur$  is a Hilbert-Schmidt operator on $A^{2}$ for some $\alpha>-1$, then for any $-1<\alpha'<\alpha$,  $\mathcal {H}_{\mu}^{\alpha'}$  is a Hilbert-Schmidt operator on $A^{2}$.
\end{cor}

\begin{re}
It is not difficult to obtain the results of compactness in Theorem \ref{th3.2}, Theorem \ref{th3.7} and Theorem \ref{th3.8}.  They are  vanishing Carleson type measures. Note that the  boundedness (resp. compactness) of  $\hu$ acting from $H^{p}(0<p\leq 1)$ to $H^{1}$ has been studied in \cite{H14,H3}. The proofs  of  Theorem \ref{th3.2} and Theorem \ref{th3.7}  can be applied to study
  the operator $\hur$ acting from $H^{p}(0<p\leq1)$ into $H^{q}(q\geq 1)$, or  $\hur$ acting from $H^{p}(0<p\leq1)$ into $A^{q}(q \geq 1)$.
 \end{re}

\section{Generalized  Hilbert  operator acting  from   $\ap$  to  $\mathcal {B}$}

\ \ \ The Bloch space $\B$ can be viewed as the limit case of $A^{q}$ as  $q\rightarrow +\infty$. In order to obtain the boundedness of   $\hur:\ap \rightarrow \B$ for all $\alpha>-1$ and $0<p<\infty$, let us do some preparations.

In \cite{gi1}, a sequence $\{V_{n}\}$ was constructed in the following way:
Let $\psi$ be a $C^{\infty}$-function on $\mathbb{R}$ such that
(1)\ $\psi(s)=1$ for $s\leq 1$, (2)\ $\psi(s)=0 $ for $s\geq 2$,  (3)\ $\psi$ is decreasing and positive on the interval $(1,2)$.

Let $\varphi(s)=\psi(\frac{s}{2})-\psi(s)$, and let $v_{0}=1+z$, for $n\geq 1$,
$$V_{n}(z)=\sk \varphi(\frac{k}{2^{n-1}})z^{k}=\sum_{k=2^{n-1}}^{2^{n+1}-1}\varphi(\frac{k}{2^{n-1}})z^{k}.$$
The polynomials $V_{n}$ have the properties:
\\ (1)\ $\displaystyle{f(z)=\sn V_{n}\ast f(z)}$ , for $f\in \hd$;
\\ (2)\ $||V_{n}\ast f||_{p}\lesssim ||f||_{p}$, for $f\in H^{p}, p>0$;
\\ (3)\ $||V_{n}||_{p}\asymp 2^{n(1-\frac{1}{p})}$, for all $p>0$,
where $\ast$ denotes the Hadamard product and $||\cdot||_{p}$ denotes the norm of Hardy space $H^{p}$.

The following  lemma  can be found in \cite[  Theorem  3.1 ]{gi2}.
\begin{lemma}\label{lm4.1}
Let  $f\in H(\mathbb{D})$, then
$f\in \mathcal {B}$ if and only if  $$\sup_{n\geq0}||V_{n}\ast f||_{\infty}<\infty.$$
Moreover, $$||f||_{\mathcal {B}}\asymp \sup_{n\geq0}||V_{n}\ast f||_{\infty}.$$
\end{lemma}
The sublinear  generalized integral type Hilbert operator $\widetilde{\ii}_{\mu_{\alpha+1}}$  is defined by
$$\widetilde{\ii}_{\mu_{\alpha+1}}(f)(z)=\I \frac{|f(t)|}{(1-tz)^{\alpha+1}}d\mu(t), \ \ \ (\alpha>-1).$$

 \begin{thm}\label{th4.2}
 Let  $0<p <\infty$ and $\alpha>-1$. Suppose  $\mu$ is  a positive Borel measure on $[0, 1)$ and satisfies  the conditions of Lemma \ref{lm3.1}. Then the following statements are equivalent:
 \\ (a)  $\hur: A^{p}\rightarrow\mathcal {B}$ is bounded;
 \\ (b) $\widetilde{\ii}_{\mu_{\alpha+1}}: A^{p}\rightarrow\mathcal {B}$ is bounded;
 \\(c)  $\mu$ is a $\frac{2}{p}+\alpha+1$ Carleson measure.
\end{thm}

\begin{proof}
 $(a)\Rightarrow (c)$. A similar discussion to the proof of Theorem \ref{th3.2} implies that
 $\hur:\ap \rightarrow \B$ is bounded if and only if
     \begin{equation}\label{eq5.1}
     \left|\I R^{0,\alpha-1} g(t)  \overline{f(t)}d\mu(t)\right|\lesssim ||f||_{\ap}||g||_{A^{1}}  \end{equation}
for all $f\in \ap$, $g\in A^{1}$.

Let $\beta>0$,  take  the test functions
 $$f_{a}(z)=\frac{(1-a^{2})^{\beta}}{(1-az)^{\beta+\frac{2}{p}}}\ \  \mbox{and}\ \ g_{a}(z)=\frac{(1-a^{2})^{\beta}}{(1-az)^{\beta+2}}, \ \ a\in (0,1).$$
 Then $f_{a}\in \ap$, $ g_{a}\in A^{1}$ and
 $$\sup_{0<a<1}||f_{a}||_{\ap}\asymp 1,  \sup_{0<a<1}||g_{a}||_{A^{1}}\asymp 1.$$
  Note that
     $$R^{0,\alpha-1} g_{a}(t)=(1-a^{2})^{\beta} \sn\frac{\Gamma(n+1+\alpha)\Gamma(n+\beta+2)}{\Gamma(n+2)\Gamma(\alpha+1)\Gamma(n+1)\Gamma(\beta+2)}(at)^{n}.$$
     By  Stirling's formula we have that
      \begin{equation}\label{eq5.2}
      R^{0,\alpha-1} g_{a}(t)\asymp \frac{(1-a^{2})^{\beta} }{(1-at)^{\beta+\alpha+1} }. \end{equation}
      Using (\ref{eq5.1}) and (\ref{eq5.2}), we have
       \[ \begin{split}
      1&\gtrsim \sup_{0<a<1}||f||_{\ap}\sup_{0<a<1}||g||_{A^{1}}\gtrsim \left|\I R^{0,\alpha-1} g_{a}(t)  \overline{f_{a}(t)}d\mu(t)\right|\\
      &\gtrsim \int_{a}^{1} R^{0,\alpha-1} g_{a}(t)\frac{(1-a)^{\beta}}{(1-at)^{\beta+\frac{2}{p}}}  d\mu(t)\\
     & \gtrsim \frac{\mu([a,1)) }{(1-a^{2})^{\frac{2}{p}+\alpha+1}}.
           \end{split} \]
  This implies that $\mu$  is a  $\frac{2}{p}+\alpha+1$-Carleson measure.


$(c)\Rightarrow (b)$.  Assume (c), then $\frac{d\mu(t)}{(1-t)^{\frac{2}{p}}}$ is an  $\alpha+1$-Carleson measure by Lemma \ref{lm2.5}. Therefore, for each $n\in \mathbb{N}$, we have
 \begin{equation}\label{eq5.3}
\I t^{n}\frac{d\mu(t)}{(1-t)^{\frac{2}{p}}}=O(\frac{1}{n^{\alpha+1}}).  \end{equation}
  For every $0 \not\equiv f\in \ap$, our assumptions guarantee that $\widetilde{\ii}_{\mu_{\alpha+1}}(f)$
 converges absolutely
 for every $z\in \dd$ and
 $$
\widetilde{\ii}_{\mu_{\alpha+1}}(f)(z)=\I \frac{|f(t)|}{(1-tz)^{\alpha+1}}d\mu(t)=\sum_{n=0}^{\infty}c_{n}z^{n},
 $$
 where  $$c_{n}=\frac{\Gamma(n+1+\alpha)}{\Gamma(n+1)\Gamma(\alpha+1)}\I t^{n}|f(t)|d\mu(t).$$
 Obviously, $\{c_{n}\}_{n=1}^{\infty}$ is   a  nonnegative sequence.  Theorem 3.3.1  in \cite{b4} or \cite[Corollary 3.2]{t1} shows  that
\begin{equation}\label{eq5.4}||\widetilde{\ii}_{\mu_{\alpha+1}}(f)||_{\mathcal {B}}\asymp |c_{0}|+ \sup_{n\geq 1}\frac{1}{n}\sum_{k=1}^{n} kc_{k}.\end{equation}
  Using Stirling's  formula, (\ref{eq4.1}) and (\ref{eq5.3}), we deduce that
  \[ \begin{split}
   |c_{0}|+ \sup_{n\geq 1}\frac{1}{n}\sum_{k=1}^{n} kc_{k}
    &\lesssim  ||f||_{\ap}+  ||f||_{\ap} \sup_{n\geq 1} \frac{1}{n} \sum_{k=1}^{n}k^{\alpha+1}\I t^{k} \frac{d\mu(t)}{(1-t)^{\frac{2}{p}}}\\
    &\lesssim ||f||_{\ap}+  ||f||_{\ap}\sup_{n\geq 1}\frac{1}{n}\sum_{k=1}^{n}k^{\alpha+1} \frac{1}{k^{\alpha+1}}\\
     &\lesssim ||f||_{\ap}.
     \end{split} \]
Hence,  $\widetilde{\ii}_{\mu_{\alpha+1}}: \ap\rightarrow \mathcal {B}$ is bounded.

$(b)\Rightarrow (a)$.  The assumption of the measure $\mu$  guarantee  that $\hur (f)=\iu (f)$ for every $f\in \ap$.
 If  $\widetilde{\ii}_{\mu_{\alpha+1}}: \ap \rightarrow \B$ is bounded, for each $f\in \ap$,  by Lemma \ref{lm4.1} we get
     $$\sup_{n\geq 1}||V_{n}\ast \widetilde{\ii}_{\mu_{\alpha+1}}(f)||_{\infty}\asymp ||\widetilde{\ii}_{\mu_{\alpha+1}}(f)||_{\B}\lesssim ||f||_{\ap} ||\widetilde{\ii}_{\mu_{\alpha+1}}||.$$
     Since the coefficients of $\widetilde{\ii}_{\mu_{\alpha+1}}(f)$ are non-negative, it is easy to check that
     $$M_{\infty}(r,V_{n}\ast \ii_{\mu_{\alpha+1}}(f) )\leq M_{\infty}(r,V_{n}\ast \widetilde{\ii}_{\mu_{\alpha+1}}(f) ) \ \
     \mbox{for all}\  0<r<1 .$$
     Therefore, $$||V_{n}\ast \ii_{\mu_{\alpha+1}}(f)||_{\infty}=\sup_{0<r<1}M_{\infty}(r,V_{n}\ast \ii_{\mu_{\alpha+1}}(f) )\leq  ||V_{n}\ast \widetilde{\ii}_{\mu_{\alpha+1}}(f)||_{\infty}.$$
Consequently,
$$||\hur (f)||_{\B}=||\ii_{\mu_{\alpha+1}}(f)||_{\B}\asymp\sup_{n\geq 1}||V_{n}\ast \ii_{\mu_{\alpha+1}}(f)||_{\infty} \lesssim ||f||_{\ap}.$$
This implies that
 $\hur: \ap\rightarrow \B$ is bounded.
\end{proof}
\begin{re}
The assumption of the measure $\mu$  also  insures that  $\widetilde{\ii}_{\mu_{\alpha+1}}$
 is well defined on $\ap$ for all $0<p<\infty$.
\end{re}

 \begin{thm}\label{th4.4}
 Let  $0<p <\infty$ and $\alpha>-1$. Suppose  $\mu$ is  a positive Borel measure on $[0, 1)$ and satisfies  the conditions of Lemma \ref{lm3.1}. Then the following statements are equivalent:
 \\ (a)  $\hur: A^{p}\rightarrow\mathcal {B}$ is compact;
 \\ (b) $\widetilde{\ii}_{\mu_{\alpha+1}}: A^{p}\rightarrow\mathcal {B}$ is compact;
 \\(c)  $\mu$ is a vanishing $\frac{2}{p}+\alpha+1$ Carleson measure.
\end{thm}
\begin{proof}
It is obvious that $(b)\Rightarrow (a)$.

$(a)\Rightarrow (c)$.  Let $\{a_{n}\}\subset [0,1)$  be any sequence with $a_{n}\rightarrow 1$.  Take
$$f_{n}(z)=\frac{(1-a_{n}^{2})^{\beta}}{(1-a_{n}z)^{\beta+\frac{2}{p}}}, \ \ \mbox{where}\  \beta >0.$$
Then $\{f_{n}\}\subset\ap$, $\sup_{n\geq 1}||f_{n}||_{\ap}\asymp 1$ and $\{f_{n}\}$ converges to $0$ uniformly on every compact subset of $\dd$.
Using Lemma  \ref{lm2.6} and the fact that $\hur: A^{p}\rightarrow\mathcal {B}$ is compact, we have that
$$\lim_{n\rightarrow \infty}||\hur(f_{n})||_{\mathcal {B}}=0.$$
This and (\ref{eq5.1}) imply that
$$\lim_{n\rightarrow \infty}\left|\int_{0}^{1}R^{0,\alpha-1}g(t)f_{n}(t)d\mu(t)\right|=0\ \ \mbox{for all}\  g \in A^{1}.$$
Now, let
$$g_{n}(z)=\frac{(1-a_{n}^{2})^{\beta}}{(1-a_{n}z)^{\beta+2}}.$$
It is clear  that $\{g_{n}\}\subset A^{1}$. Thus, using (\ref{eq5.2}) we have
     \[ \begin{split}
    &\ \ \ \  \left|\int_{0}^{1}R^{0,\alpha-1}g_{n}(t)f_{n}(t)d\mu(t)\right|\\
  &  \asymp  (1-a_{n}^{2})^{2\beta}\I \frac{d\mu(t)}{(1-a_{n}t)^{2\beta+\alpha+1+\frac{2}{p}}}\\
& \gtrsim  (1-a_{n}^{2})^{2\beta} \int_{a_{n}}^{1}\frac{d\mu(t)}{(1-a_{n}t)^{2\beta+\alpha+1+\frac{2}{p}}}\\
&   \gtrsim \frac{\mu([a_{n},1))}{(1-a_{n})^{\alpha+1+\frac{2}{p}}}.
    \end{split} \]
   This  together with $\{a_{n}\}$ is an arbitrary sequence on $[0,1)$  we have that
    $$\lim_{t\rightarrow 1^{-}}\frac{\mu([t,1))}{(1-t)^{\alpha+1+\frac{2}{p}}}=0.$$
    Hence, $\mu$  is a vanishing $\frac{2}{p}+\alpha+1$ Carleson measure.

    $(c)\Rightarrow (b)$.  Assume $(c)$, then $\frac{d\mu(t)}{(1-t)^{\frac{2}{p}}}$ is a vanishing  $\alpha+1$-Carleson measure by Lemma \ref{lm2.5}. Therefore, for any $\varepsilon>0$, there exists a integer $N$ such that
  \begin{equation}\label{eq5.5}
  n^{\alpha+1}\I t^{n}\frac{d\mu(t)}{(1-t)^{\frac{2}{p}}}< \varepsilon  \ \ \mbox{whenever}\ n >N . \end{equation}
     Let $\{f_{k}\}_{k=1}^{\infty}$ be a bounded
sequence in $\ap$  which converges to $0$ uniformly on every compact subset of  $\mathbb{D}$.  For each  $k\in \mathbb{N}$, we have
$$
\widetilde{\ii}_{\mu_{\alpha+1}}(f_{k})(z)=\I \frac{|f_{k}(t)|}{(1-tz)^{\alpha+1}}d\mu(t)=\sum_{n=0}^{\infty}c_{n,k}z^{n},
 $$
 where  $$c_{n,k}=\frac{\Gamma(n+1+\alpha)}{\Gamma(n+1)\Gamma(\alpha+1)}\I t^{n}|f_{k}(t)|d\mu(t).$$
 It is obvious that $\{c_{n,k}\}_{n=1}^{\infty}$ is a nonnegative sequence for each $k \in \mathbb{N}$.   It is sufficient  to prove that
 $$\lim_{k\rightarrow \infty}\left(c_{0,k}+\sup_{n\geq 1}\frac{1}{n}\sum_{j=1}^{n}jc_{j,k}\right) =0$$
 by (\ref{eq5.4}) and Lemma \ref{lm2.6}. Since $\I \frac{d\mu(t)}{(1-t)^{\frac{2}{p}}}<\infty$, there exists a $0<t_{0}<1$ such that
 \begin{equation}\label{eq5.6}\int_{t_{0}}^{1}\frac{d\mu(t)}{(1-t)^{\frac{2}{p}}}<\varepsilon. \end{equation}
  Bearing in mind that  $\{f_{k}\}_{k=1}^{\infty}$   converges to $0$ uniformly on every compact subset of  $\mathbb{D}$, so that there exists a integer $K$
 such that
  \begin{equation}\label{eq5.7}
  \sup_{t\in [0,t_{0}]}|f_{k}(t)|<\varepsilon \ \ \mbox{whenever}\ k>K. \end{equation}
Using (\ref{eq5.1}), (\ref{eq5.6}) and (\ref{eq5.7}) we have that
 \begin{equation}\label{eq5.8}
 \lim_{k\rightarrow \infty}c_{0,k}=\lim_{k\rightarrow \infty}\I |f_{k}(t)|d\mu(t)=0.  \end{equation}
 It is obvious that
 $$\sup_{n\geq 1}\frac{1}{n}\sum_{j=1}^{n}jc_{j,k}\leq \sup_{1\leq n\leq N}\frac{1}{n}\sum_{j=1}^{n}jc_{j,k}+ \sup_{n\geq N+1}\frac{1}{n}\sum_{j=1}^{n}jc_{j,k}.$$
   By Stirling's formula and (\ref{eq5.8}) we have
 \begin{equation}\label{eq5.9}
  \sup_{1\leq n\leq N}\frac{1}{n}\sum_{j=1}^{n}jc_{j,k}
\lesssim   \sup_{1\leq n\leq N}\frac{1}{n}\sum_{j=1}^{n}j^{\alpha+1}\I |f_{k}(t)|d\mu(t)\lesssim \I |f_{k}(t)|d\mu(t)\rightarrow 0, \ (k\rightarrow \infty).  \end{equation}
   On the other hand, by Stirling's formula, (\ref{eq4.1}) and (\ref{eq5.5}) we have
    \[ \begin{split}
     \sup_{n\geq N+1}\frac{1}{n}\sum_{j=1}^{n}jc_{j,k}
     & \lesssim  \sup_{n\geq N+1}\frac{1}{n}\sum_{j=1}^{N}j^{\alpha+1}\I t^{j}|f_{k}(t)|d\mu(t)\\
     &+  \sup_{n\geq N+1}\frac{1}{n}\sum_{j=N+1}^{n}j^{\alpha+1}\I t^{j}|f_{k}(t)|d\mu(t)\\
     & \lesssim  \sup_{n\geq N+1}\frac{1}{n}\sum_{j=1}^{N}j^{\alpha+1}\I|f_{k}(t)|d\mu(t)\\
     &+  \sup_{k\geq 1}||f_{k}||_{\ap}\sup_{n\geq N+1}\frac{1}{n}\sum_{j=N+1}^{n}j^{\alpha+1}\I t^{j}\frac{d\mu(t)}{(1-t)^{\frac{2}{p}}}\\
     & \lesssim \I|f_{k}(t)|d\mu(t)+ \varepsilon \sup_{k\geq 1}||f_{k}||_{\ap}  \sup_{n\geq N+1}\frac{1}{n}\sum_{j=N+1}^{n}j^{\alpha+1}\frac{1}{j^{\alpha+1}}\\
     &  \lesssim \I|f_{k}(t)|d\mu(t)+ \varepsilon \sup_{k\geq 1}||f_{k}||_{\ap}.
       \end{split} \]
       This implies that
     \begin{equation}\label{eq5.10}
     \lim_{k\rightarrow \infty} \sup_{n\geq N+1}\frac{1}{n}\sum_{j=1}^{n}jc_{j,k} =0.  \end{equation}
      Then (\ref{eq5.8})-(\ref{eq5.10}) imply that
       $$\lim_{k\rightarrow \infty}\left(c_{0,k}+\sup_{n\geq 1}\frac{1}{n}\sum_{j=1}^{n}jc_{j,k}\right) =0.$$
     Therefore, $\widetilde{\ii}_{\mu_{\alpha+1}}: A^{p}\rightarrow\mathcal {B}$ is compact.
\end{proof}

\subsection*{Declarations}
The authors declare that there are no conflicts of interest regarding the publication of this paper.
\vskip2mm
  \subsection*{Availability of data and material}
Data sharing not applicable to this article as no datasets were generated or analysed during
the current study: the article describes entirely theoretical research.

\bibliographystyle{els}

\end{document}